\documentclass[11pt,twoside,reqno]{amsart}
\RequirePackage[utf8]{inputenc}

\DeclareSymbolFont{matha}{OML}{txmi}{m}{it}
\DeclareMathSymbol{\varv}{\mathord}{matha}{118}

\usepackage{diagbox}
\usepackage[mathscr]{euscript}
\usepackage{amsmath,amsthm,amssymb,enumerate}
\usepackage{bbm}
\usepackage{mathrsfs}
\usepackage{comment}
\usepackage[T1]{fontenc}
\usepackage[toc,page]{appendix}
\usepackage[numbers]{natbib}

\numberwithin{equation}{section}

\newtheorem{theorem}{Theorem}[section]
\newtheorem{notation}{Notation}[section]

\newtheorem{openprob}{Open Problem}[section]
\newtheorem{corollary}{Corollary}[section]

\newtheorem{lemma}{Lemma}[section]
\newtheorem{cond}{Condition}[section]
\newtheorem{remark}{Remark}[section]

\newtheorem{definition}{Definition}[section]

\usepackage[textsize=small, color=magenta]{todonotes}
\newcommand{\RvdH}[1]{\todo[inline, color=magenta]{Remco: #1}}
\newcommand{\sss}{\scriptscriptstyle}
\newcommand{\eqn}[1]{\begin{equation}#1\end{equation}}
\newcommand{\eqan}[1]{\begin{align}#1\end{align}}

\newcommand {\convp}{\stackrel{\sss {\mathbb P}}{\longrightarrow}}

\newcommand{\prob}{\mathbb{P}}
\newcommand{\expec}{\mathbb{E}}
\newcommand{\1}{\mathbbm{1}}
\newcommand{\comp}{\mathscr{C}}

\newcommand{\shortversion}[1]{}

\newcommand{\MP}[1]{\todo[inline, color=yellow]{Manish: #1}}

\usepackage{tikz}
\usetikzlibrary{shapes,backgrounds, arrows}

\def\firstcircle{(0,0) circle (1cm)}
\def\secondcircle{(0:3cm) circle (1cm)}
\def\thirdcircle{(0: 6cm) circle (1cm)}
\def\fourthcircle{(0:9cm) circle (1cm)}
\usetikzlibrary{positioning}
\definecolor {processblue}{cmyk}{0.96,0,0,0}

\RequirePackage[utf8]{inputenc}
\usepackage{diagbox}
\usepackage[mathscr]{euscript}
\usepackage{enumerate}
\usepackage{color}
\usepackage{todonotes}

\newcommand{\ensymboldefinition}{$\blacktriangleleft$}

\newcounter{figno}

\newcommand{\vep}{\varepsilon}

\newcommand{\indic}[1]{\1_{\{#1\}}}
\newcommand{\indicwo}[1]{\1_{#1}}

\date{}

\title[Are giants in random digraphs `almost' local?]{Are giants in random digraphs `almost' local?}
\author{Remco van der Hofstad and Manish Pandey}

\begin{document}

\maketitle
\begin{abstract}
 Recently, the first author showed that the giant in random undirected graphs is `almost' local. This means that, under a necessary and sufficient condition, the limiting proportion of vertices in the giant converges in probability to the survival probability of the local limit. We extend this result to the setting of random digraphs, where connectivity patterns are significantly more subtle. For this, we identify the precise version of local convergence for digraphs that is needed. 
 
We also determine bounds on the number of strongly connected components, and calculate its asymptotics explicitly for locally tree-like digraphs, as well as for other locally converging digraph sequences under the `almost-local' condition for the strong giant. The fact that the number of strongly connected components is {\em not}  local once more exemplifies the delicate nature of strong connectivity in random digraphs.
\end{abstract}

\section{Introduction and main results}\label{sec-intro-main-results} 

\subsection{Introduction}\label{sec-intro}
In the realm of network and graph theory, the study of directed graphs, also known as digraphs, has gained considerable interest due to their ubiquitous presence in various real-world applications, such as social networks, transportation systems, and information flow analysis. Understanding their connectivity properties is fundamental for gaining insights into the dynamics of complex systems that are modelled using random digraphs. A strongly connected component is called a {\em strong giant} when its size is asymptotically linear in the number of vertices. The giant size is crucial for the understanding of connectivity in a network. There are other notions for graph components and subsequently giants in digraphs, sometimes called weak giants, which we discuss in more detail in Section \ref{sec-discussion}. 

In this paper, we explore the limiting connectivity structure of locally converging digraphs. Specifically, our research focuses on the limiting proportion of vertices in the largest strongly connected component, and the number of connected components, in these graphs. 

Local convergence of random graphs, first introduced in \cite{AldSte04, BenSch01}, is a well-studied notion. It describes what a graph looks like locally, as the number of vertices grows large. It is discussed in detail in \cite[Chapter 2]{Hofs24} for undirected graphs, and in \cite[Section 9.2]{Hofs24} for digraphs. When the local limit of a sequence of random graphs/digraphs is known, we are able to predict properties about the graph that are continuous with respect to the local topology. 

For undirected graphs, the number of vertices in the largest connected component, which is also called the giant, is \emph{not} continuous under the local topology, but \cite{Hofs21b} shows that this quantity is `almost' local. Indeed, the proportion of vertices in the giant converges to the survival probability of the local limit if and only if a certain condition holds. In this paper we show that the giant strongly connected component is also `almost' local for the case of random digraphs. 

The difference between the directed and the undirected case is quite delicate. One example to illustrate this subtlety is through the number of connected components. In the undirected case, the asymptotics for the number of connected components can easily be obtained since it is continuous in the local topology. In contrast, the directed case turns out to be much more complex due the stronger, and less local, notion of connectivity. In particular, it is an open question whether it is `almost' local or not, which we answer affirmatively in this paper.   

Connectivity properties of random digraphs have attracted substantial attention. The diameter of the directed configuration model is studied in \cite{MR4533728}, the size of the largest strongly connected component of a random digraph with prescribed degrees in \cite{MR2056402}, Weak components of the directed configuration model in \cite{coulson2021weak}, percolation in simple random digraphs with prescribed degrees in  \cite{van2023percolation}, and graph distances in \cite{HooOlv18}. 

\subsection{Preliminaries}\label{sec-preliminaries}
In this section, we discuss different types of connected components in digraphs, and informally introduce the notion of local convergence in directed random graphs.

Let $G = (V(G), E(G))$ be a digraph. We denote the length of the shortest directed path between $u, v\in V(G)$ by $d_G(u, v)$. We define $d_G(u, v) = \infty$ if there is no directed path between $u$ and $v$. In digraphs, the distance is not symmetric, i.e., $d_G(u, v)\neq d_G(v, u)$. 

\begin{definition}[Vertex-induced subgraph]\label{def: vertex induced graph.}
    \rm Let $G = (V(G), E(G))$ be a digraph and $U \subset V(G)$. Then, the digraph induced by the vertex set $U$, which we denote by $G[U]$, is the digraph obtained by removing all the vertices in $U^c = V(G)\setminus U$, along with the edges incident to these vertices, from $G$. 
    \hfill\ensymboldefinition
\end{definition}
\begin{definition}[Strongly connected component of a vertex]\label{def: scc}
    \rm Fix $v \in V(G)$, then we let $\mathscr{C}_v$ be the strongly connected component of $v$, defined as
\begin{align}
    \mathscr{C}_v :=G\left[\{x \in V(G)\colon d_G(x, v)<\infty,d_G(v, x)< \infty\}\right].
\end{align}
\hfill\ensymboldefinition
\end{definition}

\begin{definition}[In- and out-component, weak components of a vertex]\label{def: out and in cc}
\rm Fix $v \in V(G)$, then we let $\mathscr{C}^{\sss-}_v$ and $\mathscr{C}^{\sss+}_v$ to be the in- and out-component of $v$, respectively, defined as
\begin{align}
   \mathscr{C}^{\sss-}_v &:= G\left[\{x \in V(G)\colon d_G(x, v)<\infty\}\right] ,\\ \mathscr{C}^{\sss+}_v &:= G\left[\{x \in V(G)\colon d_G(v, x)< \infty\}\right].
\end{align}
\hfill\ensymboldefinition
\end{definition}
The strongly connected components of a digraph {\em partitions} the vertex set. Let $(\mathscr{C}_{\sss(i)})_{i \geq 1}$ denote the unique partition formed by the strongly connected components of $G_n$, ordered in such a way that $|\mathscr{C}_{\sss(1)}| \geq |\mathscr{C}_{\sss(2)}|\geq |\mathscr{C}_{\sss(3)}|\geq \cdots$, where $|\cdot|$ denote the size of the vertex set and ties are broken arbitrarily. Thus, $\mathscr{C}_{\sss(1)}$ is the largest strongly connected component(LSCC) of the graph, which we also denote by $\mathscr{C}_{\max}$.

\paragraph{\textbf{Local convergence in random digraphs.}} 
Local convergence of random digraphs describes what a digraph locally looks like from the perspective of a uniformly chosen vertex, as the number of vertices in the digraph tends to infinity. For example, the sparse Erd\H{o}s-R\'enyi random digraph, which is formed by bond percolation on the complete digraph of size $n$ with probability $p=\lambda/n$, locally looks like a Poisson branching process with parameter $\lambda$ in both directions, as $n$ tends to infinity. 

Now let us give the informal definition of local convergence. We write $X_n \overset{\sss\prob}{\to} X$ when $X_n$ converges in probability to $X$. 
\begin{definition}[Forward-backward neighbourhood]\label{def-forward-backward-nbd}
    Let $G$ be a digraph and $v \in V(G)$. Then $B_r^{\sss(G)}(v)$ is the rooted digraph with root $v$, induced by the vertex set $V( B_r^{\sss(G)}(v))$, where
    \begin{equation}
        V(B_r^{\sss(G)}(v)) = \left\{u\in V(G) \colon \min \{d_G(u, v), d_G(v, u)\} < r \right\},
    \end{equation}
    where a rooted digraph is a digraph with a distinguished vertex, called a root, defined formally in Definition \ref{def: rooted digraphs}.
 \hfill\ensymboldefinition   
\end{definition}

\begin{definition}[Digraph isomorphism]\label{def-dir-isomorphisms}
\rm Two digraphs, $G_1$ and $G_2$ are said to be directed isomorphic if there exists a bijection $\phi\colon V(G_1)\to V(G_2)$ such that $(v_1, v_2)\in E(G_1)$ precisely when $(\phi(v_1), \phi(v_2))\in E(G_1)$. This is abbreviated as $G_1\cong G_2$.
\hfill\ensymboldefinition
\end{definition}

\begin{definition}[Rooted digraph isomorphism]\label{def-rooted-dir-isomorphisms}
\rm Two rooted digraphs, $(G_1, o_1)$ and $(G_2, o_2)$ are said to be directed isomorphic if there exists a bijection $\phi\colon V(G_1)\to V(G_2)$ such that $(v_1, v_2)\in E(G_1)$ precisely when $(\phi(v_1), \phi(v_2))\in E(G_1)$. In addition to this, it should also hold that $\phi(o_1) = o_2$. This is abbreviated as $(G_1, o_1)\cong (G_2, o_2)$.
\hfill\ensymboldefinition
\end{definition}

We will rely on two types of local convergence, namely, forward-backward local weak convergence and forward-backward local convergence in probability. 
Let $o_n$ denote a uniformly chosen vertex from $V(G_n)$. Forward-backward local weak convergence means that
\begin{equation}\label{eq: local weak convergence} 
    \prob\left({B_r^{\sss(G_n)}(o_n)\cong (H, o')}\right) \to \bar{\mu}(B_r^{\sss(\bar{G})}(o)\cong (H,o')),
\end{equation}
for all rooted digraphs $(H, o')$. In \eqref{eq: local weak convergence}, $(\bar{G}, o) \sim \bar{\mu}$ is a {\em random} rooted graph, which is called the forward-backward local weak limit. For forward-backward local convergence in probability, instead, we require that
\begin{equation}\label{eq: local convergence in probability}
    \frac{1}{|V(G_n)|}\sum_{v\in V(G_n)}\mathbbm{1}_{\left\{B_r^{\sss(G_n)}(v)\cong (H, o')\right\}} \overset{\sss\prob}{\to} \mu(B_r^{\sss(G)}(o)\cong (H,o'))
\end{equation}
to hold for all rooted digraphs $(H, o')$. In \eqref{eq: local convergence in probability}, $(G, o) \sim \mu$ is a random rooted graph which is called the local limit in probability (and bear in mind that $\mu$ can possibly be a random measure on rooted graphs). Both \eqref{eq: local weak convergence} and \eqref{eq: local convergence in probability} describe the convergence of the proportions of vertices around which the digraph locally looks like a certain specific digraph. 

\begin{remark}[Other notions of local convergence of digraphs]\label{rem: local convergence digraphs}
\rm Other notions of local convergence of digraphs exist. For example, one can consider digraphs as undirected graphs with marks on the edges that indicate their directions. Then one can define marked local convergence of such marked undirected graphs. This notion is much stronger than forward-backward local convergence. Further, sometimes only one direction is relevant. An example arises in the study of the local limit of the PageRank distribution on a digraph, as studied in \cite{GarHofLit20}. There, the in-component matters, with as extra vertex marks the out-degree of the vertices. We see that convergence of different digraph properties require different local convergence notions. Forward-backward convergence is exactly what we need for the convergence of the strong giant and the number of strongly connected components.
\hfill\ensymboldefinition
\end{remark}

We discuss the definition of forward-backward local convergence more formally in Section \ref{sec-LC-RG}. We now turn to our main results.

\subsection{Main results}\label{sec-main-results}
Our main results are divided into three subsections, the `almost' locality of the giant, the local limit of the giant, and the number of strongly connected components.  
\subsubsection{Strong giant is `almost' local}\label{sec-strong-giant}
 In this section we state a condition for the existence and size of a giant in a sequence of random digraphs that converges locally in probability in the forward-backward sense. We start by showing a generic probabilistic upper bound on the size of the strong giant on random digraph sequences converging locally in probability:
\begin{theorem}[Upper bound on size of strong giant]\label{thm: Upper bound strong giant}
Let $(G_n)_{n\geq 1}$ be a sequence of random digraphs, of size $|V(G_n)| = n$, converging locally in probability in the forward-backward sense to $(G,o) \sim \mu$. Let $\zeta = \mu(|\mathscr{C}^{\sss-}_o| = |\mathscr{C}^{\sss+}_o| = \infty)$. Then, for all $\varepsilon >0$,
\begin{equation}
    \prob\left(\frac{|\mathscr{C}_{\max}|}{n} \geq \zeta+\varepsilon\right) \to 0.
\end{equation}
\end{theorem}

\noindent
Theorem \ref{thm: Upper bound strong giant} implies that $\zeta = \mu(|\mathscr{C}^{\sss-}_o| = |\mathscr{C}^{\sss+}_o| = \infty)$ is always an upper bound on the size of the strong giant in locally converging digraph sequences.

Next, we proceed towards a matching lower bound, for which we will state a necessary and sufficient condition. Define 
\begin{equation}\label{eq: number of vertices with large out and in components}
    N^k_n := \#\{(x, y)\in V(G_n): |\mathscr{C}^{\sss-}_{x}|,|\mathscr{C}^{\sss+}_{x}|,|\mathscr{C}^{\sss-}_{y}|, |\mathscr{C}^{\sss+}_{y}|\geq k, x\nleftrightarrow y \},
\end{equation} 
which is the number of vertex pairs that have large in- and out-components but are not in the same strongly connected component. In \eqref{eq: number of vertices with large out and in components}, $x\nleftrightarrow y$ means that either $x$ is not connected to $y$, or $y$ is not connected to $x$ (i.e., $x$ and $y$ are in different strongly connected components). Then, our main result on the `almost-local' nature of the strong giant is as follows:
\begin{theorem}[Strong giant is `almost' local]\label{thm: strong giant is local} 
Let $(G_n)_{n\geq 1}$ be a sequence of random digraphs, of size $|V(G_n)| = n$, converging locally in probability in the forward-backward sense to $(G,o) \sim \mu$. 
Assume that 
\begin{equation}\label{eq: condition for giant}
\lim_{k\to \infty}\limsup_{n\to \infty}\frac{1}{n^2}\expec\left[N^k_n\right] = 0. 
\end{equation}
Then
\begin{align}\label{eq: giant size}
\frac{|\mathscr{C}_{\max}|}{n} \overset{\sss\prob}{\to} \zeta = \mu(|\mathscr{C}^{\sss-}_o| = |\mathscr{C}^{\sss+}_o| = \infty), \qquad \frac{|\mathscr{C}_{\sss(2)}|}{n} \overset{\sss\prob}{\to} 0.
\end{align}
Under the assumption of forward-backward local convergence in probability, the condition in \eqref{eq: condition for giant} is necessary and sufficient for \eqref{eq: giant size} to hold.
\end{theorem} 

\begin{remark}[The limiting value]\label{rem-limiting-giant}
\rm One might have expected Theorem \ref{thm: strong giant is local} to hold under different conditions with a different limit. Indeed, one may have thought that $|\mathscr{C}_{\max}|/n \overset{\sss\prob}{\to} \mu(|\mathscr{C}_o| =\infty)$, which might be true in some cases. However, in many directed random graphs, the local limit is a tree in that $\mathscr{C}^{\sss-}_o$ and $\mathscr{C}^{\sss+}_o$ are disjoint trees $\mu-$almost surely. In such cases, $\mu(|\mathscr{C}_o| =\infty)=0$, while $|\mathscr{C}_{\max}|/n \overset{\sss\prob}{\to}\zeta>0$ can occur. This exemplifies the subtleties of strongly connected components in the directed setting. It would be of interest to find examples for which $\mu(|\mathscr{C}^{\sss-}_o| = |\mathscr{C}^{\sss+}_o| = \infty)=\mu(|\mathscr{C}_o| =\infty).$
\hfill\ensymboldefinition
\end{remark}

\begin{remark}[Comparison to undirected setting]
\rm The subtleties described in Remark \ref{rem-limiting-giant} are not present in the undirected setting, where $|\mathscr{C}_{\max}|/n \overset{\sss\prob}{\to} \mu(|\mathscr{C}_o| =\infty)$ and $\mathscr{C}_o$ is the undirected connected component in $(G,o)$. This can be understood by noting that if $u\in B_r^{\sss(G)}(v)$ in the undirected setting, then we know that $u\in \comp_v$, while if $u\in B_r^{\sss(G)}(v)$ in the directed setting, then we do not know that $u\in \comp_v$.
\hfill\ensymboldefinition
\end{remark}

\begin{remark}[Applications of Theorem \ref{thm: strong giant is local}]
\rm Theorem \ref{thm: strong giant is local} can be applied to the directed Erd\H{o}s-R\'enyi random graph model, and the directed configuration model with appropriate regularity conditions on the degree distribution. Condition \eqref{eq: condition for giant} for these models can be shown using coupling arguments similar to the ones 
used in \cite[Section 2.3]{Hofs24} and \cite{Hofs21b}, respectively. We do not provide more details.
\hfill\ensymboldefinition
\end{remark}

\begin{remark}[Simplification of condition in Theorem \ref{thm: strong giant is local}]
\rm It is possible to simplify the conditions in Theorem \ref{thm: strong giant is local}. Indeed, define
\begin{equation}
\label{eq: number of vertices with large out and in components version 2}
    N^k_n(2) = \#\{(x, y)\in V(G_n): |\mathscr{C}^{\sss-}_{x}|,|\mathscr{C}^{\sss+}_{x}|,|\mathscr{C}^{\sss-}_{y}|, |\mathscr{C}^{\sss+}_{y}|\geq k, x\nrightarrow y \},
\end{equation}
where $x \nrightarrow y$ means that there is no directed path from $x$ to $y$. In Lemma \ref{lem: relation between two counts} below, we show that $N^k_n(2) \leq N^k_n \leq 2N^k_n(2)$. Thus, one can replace $N^k_n$ by $N^k_n(2)$ in \eqref{eq: condition for giant}.
\hfill\ensymboldefinition
\end{remark}

\subsubsection{Local limit of the giant}\label{sec-local-limit-giant} In this section, we discuss the local convergence of the giant. Denote the number of vertices $v$ in the giant component satisfying $(d_v^{\sss-}, d_v^{\sss+}) = (l,m)$ by $v_{1}(l, m)$. Then, we have the following asymptotics for $v_{1}(l, m)$:
\begin{theorem}[Limiting degree distribution of the giant]\label{thm: given degree vertices in giant} Under the assumptions of Theorem \ref{thm: strong giant is local},
    \begin{equation}\label{eq: First}
        \frac{v_{1}(l, m)}{n} \overset{\sss\prob}{\to} \mu\left(|\mathscr{C}^{\sss-}_o| = |\mathscr{C}^{\sss+}_o| = \infty, D_o=(l, m)  \right).
    \end{equation}
 Further, when $D^{\sss-}_{o_n}$ and $D^{\sss+}_{o_n}$ are both uniformly integrable, 
 \begin{equation}\label{eq: second}
        \frac{|E(\mathscr{C}_{\max})|}{n} \overset{\sss\prob}{\to} \frac{1}{2}\expec[\left(D_o^{\sss-}+D_o^{\sss+}\right)\mathbbm{1}_{\{|\mathscr{C}^{\sss-}_o| = |\mathscr{C}^{\sss+}_o| = \infty\}}].
 \end{equation}
\end{theorem}

\noindent
The next theorem tells us what the strong giant and its complement look like locally:
\begin{theorem}[Local limit of the strong giant and its complement]\label{thm: local limit of giant scc}
    Under the assumptions of Theorem \ref{thm: strong giant is local},
    \begin{equation}\label{eq: local convergence of giant}
        \frac{1}{n} \sum_{v \in \comp_{\max}}\1_{\left\{B_r^{\sss(G_n)}(v) \cong H_\star\right\}} \overset{\sss\prob}{\longrightarrow} \mu\left(|\mathscr{C}^{\sss-}_o|=|\mathscr{C}^{\sss+}_o|=\infty, B_r^{\sss(G)}(o) \cong H_\star\right), 
    \end{equation}
    and 
    \begin{equation}\label{eq: local convergence of complement of giant}
        \frac{1}{n} \sum_{v \notin \comp_{\max}}\1_{\left\{B_r^{\sss(G_n)}(v) \cong H_\star\right\}} \overset{\sss\prob}{\longrightarrow} \mu\left(|\mathscr{C}^{\sss-}_o|=|\mathscr{C}^{\sss+}_o|<\infty, B_r^{\sss(G)}(o) \cong H_\star\right). 
    \end{equation}
\end{theorem}

\subsubsection{Number of strongly connected components} In this section, we discuss the number of strongly connected components of digraph sequences having a local limit:

\begin{theorem}[Limiting number of SCCs]\label{thm: limiting number of scc}
    Under the assumptions of Theorem \ref{thm: strong giant is local}, and further assuming that $|\comp_{\max}|\convp \infty$ when $\zeta=0$,
    \begin{equation}
    \label{eq:number sccs almost local}
        \frac{K_n}{n} \convp \expec_\mu\Bigg[\frac{1}{|\comp(o)|}\indicwo{\{|\comp^-(o)|<\infty\}\cup \{|\comp^+(o)|<\infty\}}\Bigg].
    \end{equation}
\end{theorem}

\begin{remark}[Comparison to undirected setting]
\rm In the undirected setting, $K_n/n\convp \expec_\mu[1/|\comp(o)|]$ always holds when $G_n$ converges locally in probability, i.e., we do not need to rely on the additional `giant-is-almost-local condition' that is the undirected equivalent of \eqref{eq: condition for giant}. This once again shows that strong connectivity in digraphs is a more delicate notion than connectivity in undirected graphs.
Having Theorem \ref{thm: strong giant is local} in mind, we can think of 
    \[
    \indicwo{\{|\comp^-(o)|<\infty\}\cup \{|\comp^+(o)|<\infty\}}=
    1-\indicwo{\{|\comp^-(o)|=\infty\}\cap \{|\comp^+(o)|=\infty\}}
    \]
as enforcing that $o$ is in the limit of the strong giant.\hfill\ensymboldefinition
\end{remark}

We next give a simpler limiting result for locally tree-like digraphs, which we first define:
\begin{definition}[Locally tree-like digraphs]
   \rm We say that a sequence of digraphs $(G_n)_{n\geq 1}$ is \emph{locally tree-like} when the sequence converges locally in probability to $(G, o)\sim \mu$, which is almost surely a random directed tree.  By directed tree we mean that the strongly connected component of each vertex is almost surely the vertex itself i.e., a graph with no directed cycles. Directed configuration models, directed Erd\H{o}s-R\'enyi models, etc.\ are all locally tree-like digraphs.\hfill\ensymboldefinition
\end{definition}

\begin{theorem}[Number of SCCs in locally tree-like digraphs]\label{thm: Number of scc in locally tree-like graph}
Let $(G_n)_{n\geq 1}$ be a sequence of random digraphs, of size $|V(G_n)| = n$, converging locally in probability in the forward-backward sense to $(G,o) \sim \mu$, which is a random directed tree. Then, under the assumption of Theorem \ref{thm: strong giant is local},
    \begin{equation}\label{eq: limiting no scc for tree-like graph}
        \frac{K_n}{n} \overset{\sss\prob}{\longrightarrow} 1-\zeta.
    \end{equation}
\end{theorem}

\begin{remark}[Locally tree-like digraphs are special]
    \rm The limiting value for the number of strongly connected components in Theorem \ref{thm: limiting number of scc}, if computed for locally tree-like digraphs, matches with the expression in Theorem \ref{thm: Number of scc in locally tree-like graph}. Thus, Theorem \ref{thm: Number of scc in locally tree-like graph} tells us that we do not need additional conditions for locally tree-like digraphs, as in Theorem \ref{thm: limiting number of scc}. 
    \hfill\ensymboldefinition
\end{remark}

\begin{remark}[Upper bound in terms of $\zeta$]
\rm The upper bound in Theorem \ref{thm: Number of scc in locally tree-like graph} holds more generally. Indeed, the bound
    \eqn{
    \frac{K_n}{n} \leq 1-\frac{|\comp_{\max}|}{n}+\frac{1}{n}
    }
always holds. Thus, under the assumptions of Theorem \ref{thm: Number of scc in locally tree-like graph}, whp for every $\vep>0$,
    \eqn{
    \frac{K_n}{n} \leq 1-\zeta+\vep.
    }
\hfill\ensymboldefinition
\end{remark}

\begin{remark}[Comparison of proofs between directed and undirected graphs]
\rm  The complexity of proofs in digraphs in comparison to undirected graphs varies depending on the connectivity property we are investigating. In particular, the proofs for the size of the strong giant are similar to the proofs used for the undirected case in \cite{Hofs21b}. In contrast, for digraphs, the number of connected components is a non-local quantity, whereas, for undirected graphs, it is local. Thus, the proof becomes more complicated and less general. 
\hfill\ensymboldefinition
\end{remark}

\section{Local convergence in digraphs}\label{sec-LC-RG}
In this section, we formally define the notion of forward-backward local weak convergence and local convergence in probability for digraphs. There are several versions of local convergence in random digraphs. Amongst these, the four most commonly discussed are forward, backward, forward-backward, and marked (see Remark \ref{rem: local convergence digraphs}). 

We start by defining rooted digraphs:
\begin{definition}[Rooted digraph]\label{def: rooted digraphs}
    \rm A digraph is $G = (V(G), E(G))$ together with a vertex $o\in V(G)$ is called a rooted digraph. We denote a rooted digraph $G$ with root $o\in V(G)$ by $(G, o)$.
    \hfill\ensymboldefinition
\end{definition}

Next we define the forward-backward metric space over the rooted digraphs: 
\begin{definition}[Metric space]\label{def: Metric Space on Marked rooted graphs}
 \rm Let $\mathscr{G}_\star$ be the space of all rooted digraphs. Let $(G_1, o_1),$ $(G_2, o_2) \in \mathscr{G}_\star$. Define the metric $d_{\sss \mathscr{G}_\star}\colon \mathscr{G}_\star \times \mathscr{G}_\star \to \mathbb{R}_{\geq 0}$ by
\begin{equation}
     d_{\sss \mathscr{G}_\star}((G_1, o_1), (G_2, o_2)) := \frac{1}{R^\star+1}.
\end{equation}
 where $R^\star$ is given by
  \begin{align*}
        R^\star := &\sup\{r\geq 0\text{ s.t. }B^{\sss(G_1)}_r(o_1) \cong B^{\sss(G_2)}_r(o_2)\},
  \end{align*}
where $\cong$ denotes rooted digraph isomorphism defined in Definition \ref{def-rooted-dir-isomorphisms}.
\hfill\ensymboldefinition
\end{definition}
\begin{remark}[Equivalence classes]
\rm  Obviously, $R^\star=\infty$ when $(G_1, o_1)\cong (G_2,o_2)$ so that $d_{\sss \mathscr{G}_\star}((G_1, o_1), (G_2, o_2))=0.$ Therefore, $d_{\sss \mathscr{G}_\star}$ defines a metric on the equivalence classes of isomorphic digraphs.
\hfill\ensymboldefinition
\end{remark}
\begin{definition}[Forward-backward local convergence in probability of digraphs]
   \rm  The sequence $(G_n)_{n\geq 1}$ of random digraphs is said to converge locally in probability to the random rooted digraph $(G, o)$, a random variable taking values in $\mathscr{G}_\star$ having law $\mu$, as $n\to \infty$, if for every continuous and bounded function $f\colon \mathscr{G}_\star \to \mathbb{R}$ in the forward-backward metric,
     \begin{equation}
         \mathbb{E}_n[f(G_n, o_n)] \overset{\sss\prob}{\longrightarrow} \mathbb{E}_{\mu}[f(G, o)],
     \end{equation}
     where $o_n$ is a uniformly chosen vertex from $V(G_n)$.
     \hfill\ensymboldefinition
\end{definition}

\section{Giants in converging digraphs: Proofs of Theorems \ref{thm: Upper bound strong giant} \& \ref{thm: strong giant is local}}\label{sec-giants-of-locally-converging-digraphs}
 In this section we derive a condition for the existence of a giant in a sequence of random digraphs that converges locally in probability in the forward-backward sense. We start by proving the generic probabilistic upper bound on the size of the strong giant in Theorem \ref{thm: Upper bound strong giant}.
\begin{proof}[Proof of Theorem \ref{thm: Upper bound strong giant}]
Let $Z_{\geq k}$ be the number of vertices in $G_n$ having in- and out-components size at least $k$, i.e., 
\begin{equation}\label{eq: counting vertices with large in- and out component}
    Z_{\geq k} = \sum_{v \in V(G_n)} \mathbbm{1}_{\{ |\mathscr{C}^{\sss-}_v|\geq k, |\mathscr{C}^{\sss+}_v|\geq k\}}.
\end{equation}
Let $\zeta_{\geq k} = \mu(|\mathscr{C}^{\sss-}_o|\geq k, |\mathscr{C}^{\sss+}_o| \geq k)$. Then, taking $k\to\infty$, we have $\zeta_{\geq k} \to\mu(|\mathscr{C}^{\sss-}_o| = |\mathscr{C}^{\sss+}_o| = \infty)$. 

Notice that $(G, o) \mapsto \mathbbm{1}_{\{ |\mathscr{C}^{\sss-}_v|\geq k, |\mathscr{C}^{\sss+}_v|\geq k\}}$ is a bounded continuous function in the forward-backward sense of random digraphs, and thus due to the forward-backward local convergence in probability, $Z_{\geq k}/n \convp \zeta_{\geq k}$. Therefore, for all $\varepsilon>0$,
\begin{equation}
   \prob\left(\Big|\frac{Z_{\geq k}}{n} - \zeta_{\geq k}\Big|>\varepsilon\right)\to 0.
\end{equation}
Also, $\zeta_{\geq k} \to \zeta$ as $k \to \infty$, which means that there exists $K\in \mathbb{N}$ such that $\zeta_{\geq k} \leq \zeta +\varepsilon/2$ for all $k \geq K$. Further, $\left\{|\mathscr{C}_{\max}| \geq x\right\} \subseteq \left\{Z_{\geq k} \geq x\right\}$, for every $x\geq 1$. Thus,
\begin{equation}
    \prob\left(\frac{|\mathscr{C}_{\max}|}{n} \geq \zeta+\varepsilon\right) \leq \prob\left(\frac{|\mathscr{C}_{\max}|}{n} \geq \zeta_{\geq k}+\frac{\varepsilon}{2}\right) \leq \prob\left(\frac{Z_{\geq k}}{n} \geq \zeta_{\geq k}+\frac{\varepsilon}{2}\right)\to 0.
\end{equation}
This completes the proof.
\end{proof}
\begin{remark}[In- and out-components of strongly connected components]\label{rem: out and in same size in scc}
    \rm Each vertex in a strongly connected component has the same in- and out-component. This is illustrated in Figure \ref{Fig: bowtie} for the largest strongly connected component. We denote the in- and out-component associated with $\comp_{\sss(i)}$ by $\mathscr{C}^{\sss-}_{\sss(i)}$ and $\mathscr{C}^{\sss+}_{\sss(i)}$, respectively. 
    \hfill\ensymboldefinition
\end{remark}
In the remainder of the proof, the following notation will be essential:
\begin{notation}\label{notation: probability convergence}\rm Suppose $(X_n)_{n\geq 1}$ and $(X_{n, k})_{n\geq 1, k\geq 1}$ are sequences of random variables. We write $X_{n, k}=o_{k, \sss\prob}(X_n)$ when 
        $$\lim_{k\to \infty}\limsup_{n\to \infty}\prob\left(|X_{n, k}| >\varepsilon |X_n| \right) = 0.$$
        \hfill\ensymboldefinition
\end{notation}

\begin{lemma}[Sum of squares of cluster sizes]\label{lem: sum of squares cluster sizes}
    Under the assumptions of Theorem \ref{thm: strong giant is local},
 \begin{equation}\label{eq: second part of lemma}
        \frac{1}{n^2}\sum_{i\geq 1}|\mathscr{C}_{(i)}|^2\mathbbm{1}_{\{ |\mathscr{C}^{\sss-}_{(i)}|\geq k, |\mathscr{C}^{\sss+}_{(i)}|\geq k\}} = \zeta^2 + o_{k, \sss\prob}(1).
    \end{equation}
\end{lemma}
\begin{proof}
By Remark \ref{rem: out and in same size in scc}, we can rewrite $Z_{\geq k}$ in \eqref{eq: counting vertices with large in- and out component} as 
\begin{equation}\label{eq: new equation for large scc}
   Z_{\geq k} = \sum_{v \in V(G_n)} \mathbbm{1}_{\{ |\mathscr{C}^{\sss-}_v|\geq k, |\mathscr{C}^{\sss+}_v|\geq k\}} = \sum_{i \geq 1}|\mathscr{C}_{(i)}|\mathbbm{1}_{\{ |\mathscr{C}^{\sss-}_{(i)}|\geq k, |\mathscr{C}^{\sss+}_{(i)}|\geq k\}}. 
\end{equation}
Squaring \eqref{eq: new equation for large scc} gives
\begin{equation}\label{eq: squaring}
   Z_{\geq k}^2 = \sum_{i \geq 1}|\mathscr{C}_{(i)}|^2\mathbbm{1}_{\{ |\mathscr{C}^{\sss-}_{(i)}|\geq k, |\mathscr{C}^{\sss+}_{(i)}|\geq k\}} + \sum_{i\neq j}|\mathscr{C}_{(i)}||\mathscr{C}_{(j)}|\mathbbm{1}_{\{ |\mathscr{C}^{\sss-}_{(i)}|, |\mathscr{C}^{\sss+}_{(i)}|,|\mathscr{C}^{\sss-}_{(j)}|, |\mathscr{C}^{\sss+}_{(j)}|\geq k\}}. 
\end{equation}
The second cross-product term counts the number of pairs of vertices that are not in the same strongly connected component, but have at least $k$ vertices in their in- and out-components. Thus, by \eqref{eq: number of vertices with large out and in components},
\begin{equation}\label{eq: squaring part 2}
   \frac{Z_{\geq k}^2}{n^2} = \frac{1}{n^2}\sum_{i \geq 1}|\mathscr{C}_{(i)}|^2\mathbbm{1}_{\{ |\mathscr{C}^{\sss-}_{(i)}|\geq k, |\mathscr{C}^{\sss+}_{(i)}|\geq k\}} + \frac{1}{n^2}N^k_n. 
\end{equation}
Fix $\varepsilon > 0$. By the Markov inequality,
 \begin{equation}\label{eq: okp}
\lim_{k\to \infty}\limsup_{n\to \infty}\prob\left(\frac{1}{n^2}N^k_n > \varepsilon\right) \leq \lim_{k\to \infty}\limsup_{n\to \infty}\frac{1}{\varepsilon n^2}\expec\left[N^k_n\right]=  0, 
\end{equation}
Thus, $ N^k_n = o_{k, \sss\prob}(n^2)$, proving \eqref{eq: second part of lemma}. This completes the proof.
\end{proof}

Next we prove Theorem \ref{thm: strong giant is local}:
\begin{proof}[Proof of Theorem \ref{thm: strong giant is local}] Due to  Theorem \ref{thm: Upper bound strong giant}, it is sufficient to show \eqref{eq: giant size} for $\zeta > 0$. Thus, we assume that $\zeta>0$.  Observe that
 \begin{equation}\label{eq: using sum of square to bound giant from below}
        \frac{1}{n^2}\sum_{i\geq 1}|\mathscr{C}_{(i)}|^2\mathbbm{1}_{\{ |\mathscr{C}^{\sss-}_{(i)}|\geq k, |\mathscr{C}^{\sss+}_{(i)}|\geq k\}} \leq \frac{|\mathscr{C}_{\max}|}{n}\frac{1}{n}\sum_{i\geq 1}|\mathscr{C}_{(i)}|\mathbbm{1}_{\{ |\mathscr{C}^{\sss-}_{(i)}|\geq k, |\mathscr{C}^{\sss+}_{(i)}|\geq k\}},
    \end{equation}
which means that 
\begin{equation}\label{eq: using sum of square to bound giant from below part 1}
         \frac{|\mathscr{C}_{\max}|}{n} \geq \frac{\frac{1}{n^2}\sum_{i\geq 1}|\mathscr{C}_{(i)}|^2\mathbbm{1}_{\{ |\mathscr{C}^{\sss-}_{(i)}|\geq k, |\mathscr{C}^{\sss+}_{(i)}|\geq k\}}}{\frac{1}{n}\sum_{i\geq 1}|\mathscr{C}_{(i)}|\mathbbm{1}_{\{ |\mathscr{C}^{\sss-}_{(i)}|\geq k, |\mathscr{C}^{\sss+}_{(i)}|\geq k\}}} = \frac{\zeta^2 + o_{k, \sss\prob}(1)}{\zeta_{\geq k} + o_{\sss\prob}(1)}.
\end{equation}
Letting $k\to \infty$ in \eqref{eq: using sum of square to bound giant from below}, for all $\varepsilon>0$, we get
\begin{equation}\label{eq: using sum of square to bound giant from below part 2}
        \lim_{n\to \infty}\prob\left( \frac{|\mathscr{C}_{\max}|}{n} \geq \zeta - \varepsilon\right)=1.
\end{equation}
Equation \eqref{eq: using sum of square to bound giant from below part 2}, combined with Theorem \ref{thm: Upper bound strong giant}, completes the proof of the law of large numbers on the strong giant under the condition \eqref{eq: condition for giant} in Theorem \ref{thm: strong giant is local}.

Next we show that \eqref{eq: condition for giant} is also a necessary condition for the giant component to converge to $\zeta$. Suppose that the condition in \eqref{eq: condition for giant} fails. This implies that
\begin{equation}\label{eq: giant condition failure}
    \limsup_{k\to \infty}\limsup_{n\to \infty}\frac{1}{n^2}\expec[N_n^k]= \kappa >0.
\end{equation}
In turn, this implies that there exists a subsequence $(n_l)_{l\geq 1}$ such that
\begin{equation}
  \lim_{l\to \infty}\limsup_{k\to \infty} \frac{1}{n_l^2}\expec[N_{n_l}^k] = \kappa.  
\end{equation}
Thus, along the subsequence  $(n_l)_{l\geq 1}$,
\begin{align*}
    \frac{1}{n_{l}^2}\expec \left[|\comp_{\max}|^2\1_{\{|\comp_{\max}^+| \geq k, |\comp_{\max}^-| \geq k\}}\right] &\leq \frac{1}{n_l^2}\expec\left[\sum_{i\geq 1}|\comp_{\sss(i)}|^2\1_{\{|\comp_{\sss(i)}^+| \geq k, |\comp_{\sss(i)}^-| \geq k\}}\right]\\
    & = \frac{1}{n_l^2}\expec[Z_{\geq k}^2 - N_{n_l}^k] \to \zeta_{\geq k}^2 - \kappa.
\end{align*}
Notice that 
\eqan{
    \frac{1}{n_{l}^2}\expec \left[|\comp_{\max}|^2\right] &= \frac{1}{n_{l}^2}\expec \left[|\comp_{\max}|^2\1_{\{|\comp_{\max}^+| \geq k, |\comp_{\max}^-| \geq k\}}\right]\\
    &\qquad+ \frac{1}{n_{l}^2}\expec \left[|\comp_{\max}|^2\1_{\{|\comp_{\max}^+| < k\} \cup \{|\comp_{\max}^-| < k\}}\right].\nonumber
    }
We know that $|\comp_{\max}^\pm|\geq |\comp_{\max}|$. Thus, 
$$\{|\comp_{\max}^+| < k\}\cup \{|\comp_{\max}^-| < k\} \subseteq \{|\comp_{\max}| < k\},$$ 
so that, in turn,
\begin{equation}
    \frac{1}{n_{l}^2}\expec \left[|\comp_{\max}|^2\1_{\{|\comp_{\max}^+| < k\} \cup \{|\comp_{\max}^-| < k\}}\right] \leq \left(\frac{k}{n_l}\right)^2 \to 0,
\end{equation}
as $n_l \to \infty$. This implies 
\begin{equation}\label{eq: gives a contradiction}
    \lim_{l\to \infty}\frac{1}{n_l^2}\expec[|\comp_{\max}|^2] \leq \zeta^2-\kappa < \zeta^2.
\end{equation}
We conclude that $|\comp_{\max}|/n \overset{\sss\prob}{\to} \zeta$ cannot hold. This is because by the bounded convergence theorem, it would imply that $\expec\left[\left(|\comp_{\max}|/n\right)^2\right]\to \zeta^2$ which gives a contradiction with \eqref{eq: gives a contradiction}.
\end{proof}
Recall the definition of $N_n^k$ in \eqref{eq: number of vertices with large out and in components}, and that of $N_n^k(2)$ in \eqref{eq: number of vertices with large out and in components version 2}.
We now show that both are of the same order of magnitude:

\begin{lemma}[Relaxing the assumption]\label{lem: relation between two counts}
    $N^k_n(2) \leq N^k_n \leq 2N^k_n(2).$
\end{lemma}
\begin{proof} 
For $x, y \in V(G_n)$ we define the events $A_{x, y}$ and $B_{x, y}$ as
\begin{align}
    &A_{x, y}= \left[ |\mathscr{C}^{\sss-}_{x}|,|\mathscr{C}^{\sss+}_{x}|,|\mathscr{C}^{\sss-}_{y}|, |\mathscr{C}^{\sss+}_{y}|\geq k, x\nrightarrow y \right],\\
    &B_{x, y}=\left[ |\mathscr{C}^{\sss-}_{x}|,|\mathscr{C}^{\sss+}_{x}|,|\mathscr{C}^{\sss-}_{y}|, |\mathscr{C}^{\sss+}_{y}|\geq k, x\nleftarrow y \right].
\end{align}
Then,
\begin{align}
    N^k_n &=  \sum_{(x, y)\in V(G_n)^2}
    \!\!\!\!\mathbbm{1}_{A_{x, y} \cup B_{x, y}} , \qquad N^k_n(2) = \sum_{(x, y)\in V(G_n)^2}\!\!\!\!\mathbbm{1}_{A_{x, y}} = \sum_{(x, y)\in V(G_n)^2}\!\!\!\!\mathbbm{1}_{B_{x, y}}.
\end{align}
Using the fact $\mathbbm{1}_{A}\leq\mathbbm{1}_{A\cup B}\leq \mathbbm{1}_{A} +\mathbbm{1}_{B}$, we get
\begin{align}
 N^k_n(2) = \sum_{(x, y)\in V(G_n)^2}\!\!\!\!\mathbbm{1}_{A_{x, y}} \leq   N^k_n \leq \sum_{(x, y)\in V(G_n)^2}\!\!\!\!\left(\mathbbm{1}_{A_{x, y}} + \mathbbm{1}_{B_{x, y}}\right) = 2N^k_n(2),
\end{align}
which completes the proof.
\end{proof}

\section{Local limit of the giant: Proof of Theorems \ref{thm: given degree vertices in giant} \& \ref{thm: local limit of giant scc}}\label{sec-giant-local-convergence}
In this section we prove local limit properties of the giant component: 
\begin{proof}[Proof of Theorem \ref{thm: given degree vertices in giant}]
When $\zeta=0$, there is nothing to prove, so we may again assume that $\zeta>0$. For $A\subseteq \mathbb{N}^2$, we define
\begin{equation}
    Z_{A, \geq k} = \sum_{v \in V(G_n)}\mathbbm{1}_{\{|\mathscr{C}^{\sss-}_v|, |\mathscr{C}^{\sss+}_v| \geq k, (d^{\sss-}_v, d^{\sss+}_v) \in A\}}.
\end{equation}
Assuming that $G_n$ converges locally in probability in the forward-backward sense, we get
\begin{equation}
    \frac{Z_{A, \geq k}}{n} \overset{\sss\prob}{\to} \mu\left(|\mathscr{C}^{\sss-}_o|, |\mathscr{C}^{\sss+}_o| \geq k, (d^{\sss-}_o, d^{\sss+}_o) \in A\right).
\end{equation}
Since $|\mathscr{C}_{\max}|> k$ with high probability (it is here that we use that $\zeta>0$), we have 
\begin{equation}\label{eq: mu_k(A)}
    \frac{1}{n}\sum_{(l,m)\in A}v_1(l,m) \leq \frac{Z_{A, \geq k}}{n} \overset{\sss\prob}{\to} \mu\left(|\mathscr{C}^{\sss-}_o|, |\mathscr{C}^{\sss+}_o| \geq k, (d^{\sss-}_o, d^{\sss+}_o) \in A\right).
\end{equation}
Define
\begin{equation}
     \mu_k(A) = \mu\left(|\mathscr{C}^{\sss-}_o|, |\mathscr{C}^{\sss+}_o| \geq k, (d^{\sss-}_o, d^{\sss+}_o) \in A\right).
\end{equation}
Applying \eqref{eq: mu_k(A)} for $A = \{(l, m)\}^c$, we get
\begin{equation}
    \lim_{n\to\infty}\prob\left(\frac{1}{n}\left|\mathscr{C}_{\max} - v_1(l,m) \right| \leq \mu_k(A) + \frac{\varepsilon}{2} \right)=1.
\end{equation}
We argue by contradiction. Suppose that for some pair $(l, m) \in \mathbb{N}^2$, 
\begin{equation}
    \liminf_{n\to\infty}\prob\left(\frac{1}{n}v_1(l,m)  \leq \mu_k(\{l, m\}) - \varepsilon \right)= \kappa >0.
\end{equation}
Then along the subsequence $(n_j)_{j\geq 1}$, where the above liminf is attained, with a non-zero probability the following holds:
\begin{equation}\label{eq: vertices in giant}
\frac{1}{n}|\mathscr{C}_{\max}| =  \frac{1}{n}\left( |\mathscr{C}_{\max} |- v_1(l,m)\right) + \frac{1}{n}v_1(l,m) \leq \mu\left(|\mathscr{C}^{\sss-}_o|= |\mathscr{C}^{\sss+}_o| =\infty\right) - \frac{\varepsilon}{2},
\end{equation}
which contradicts Theorem \ref{thm: strong giant is local}. Thus, \eqref{eq: First} holds.

For the next part, we just observe that
\begin{equation}\label{eq: edges}
    \frac{|E(\comp_{\max})|}{n} = \frac{1}{2n}\sum_{k, l}(k+l)v_1(k, l).
\end{equation}
Thus, for any natural number $N$.
\begin{equation}\label{eq: edges in giant}
    \frac{|E(\comp_{\max})|}{n} = \frac{1}{2n}\sum_{k+l\leq N}(k+l)v_1(k, l) + \frac{1}{2n}\sum_{k+l > N}(k+l)v_1(k, l).
\end{equation}
From \eqref{eq: First} and the dominated convergence theorem, we conclude that
\eqan{\label{eq: label}
 \frac{1}{2n}\sum_{k+l\leq N}(k+l)v_1(k, l) &\overset{\sss\prob}{\longrightarrow} \frac{1}{2}\sum_{k+l\leq N}(k+l)\mu\left(|\mathscr{C}^{\sss-}_o| = |\mathscr{C}^{\sss+}_o| = \infty, D_o=(k, l)  \right)\nonumber\\
 &=\frac{1}{2}\expec_{\mu}\left[ \left( D^{\sss-}_{o}+D^{\sss+}_{o}\right)\1_{\{D^{\sss-}_{o}+D^{\sss+}_{o} < N\}}\right].
}
For the second term in \eqref{eq: edges in giant}, we bound $v_1(k, l)$ by the total number of vertices with in- degree $k$ and out-degree $l$, which we denote by $n_{k, l}$, to get
\begin{align}\label{eq: label 2}
\begin{split}
   \frac{1}{2n}\sum_{k+l > N}(k+l)v_1(k, l)& \leq \frac{1}{2}\sum_{k+l > N}(k+l)\frac{n_{k, l}}{n}\\
   &= \frac{1}{2}\expec\left[\left(D^{\sss-}_{o_n}+D^{\sss+}_{o_n}\right)\1_{\{D^{\sss-}_{o_n}+D^{\sss+}_{o_n} >N \}}\mid G_n\right].
\end{split}
\end{align}
By uniform integrability of $D^{\sss-}_{o_n}$ and $D^{\sss+}_{o_n}$,
\begin{equation}
    \lim_{N\to \infty}\limsup_{n\to \infty}\expec\left[\left(D^{\sss-}_{o_n}+D^{\sss+}_{o_n}\right)\1_{\{D^{\sss-}_{o_n}+D^{\sss+}_{o_n} >N \}}\right] = 0.
\end{equation}
Thus, using the Markov inequality we can conclude that for each $\varepsilon>0$, there exists a $N = N(\varepsilon)<\infty$ such that 
\begin{equation}
    \prob \left(\expec\left[\left(D^{\sss-}_{o_n}+D^{\sss+}_{o_n}\right)\1_{\{D^{\sss-}_{o_n}+D^{\sss+}_{o_n} >N \}}\mid G_n\right] >\varepsilon\right) \to 0.
\end{equation}
This completes the proof for the second part of the theorem.
\end{proof}

\begin{proof}[Proof of Theorem \ref{thm: local limit of giant scc}]
When $\zeta=0$, there is nothing to prove, so we may again assume that $\zeta>0$. The proof is a minor modification of that of Theorem \ref{thm: given degree vertices in giant}. Let $\mathscr{G}_\star$ be the space of all rooted digraphs. By forward-backward local convergence in probability of $(G_n)_{n\geq 1}$, for every $\mathscr{H}_\star\subseteq \mathscr{G}_\star,$ 
\begin{equation}
     \frac{1}{n} \sum_{v \in V(G_n)}\1_{\left\{B_r^{\sss(G_n)}(v) \in \mathscr{H}_\star\right\}} \overset{\sss\prob}{\longrightarrow} \mu\left( B_r^{\sss(G)}(o) \in \mathscr{H}_\star\right).
\end{equation}
Thus, \eqref{eq: local convergence of complement of giant} follows from \eqref{eq: local convergence of giant}. Therefore, we only need to show \eqref{eq: local convergence of giant}. Define
    \begin{align}
    \begin{split}
        Z_{\geq k, \mathscr{H}_\star} &=  \frac{1}{n} \sum_{v \in V(G_n)}\1_{\left\{|\mathscr{C}^{\sss-}_v|\geq k, |\mathscr{C}^{\sss+}_v|\geq k, B_r^{\sss(G_n)}(v) \in \mathscr{H}_\star\right\}}\\
        &\overset{\sss\prob}{\longrightarrow} \mu\left(|\mathscr{C}^{\sss-}_o|\geq k, |\mathscr{C}^{\sss+}_o|\geq k, B_r^{\sss(G)}(o) \in \mathscr{H}_\star\right).
    \end{split}
    \end{align}
Because of Theorem \ref{thm: strong giant is local}, we have that $|\comp_{\max}|/n \overset{\sss\prob}{\to} \zeta >0$. Thus, on the high probability event $\left\{|\comp_{\max}| \geq k\right\}$,
\begin{equation}\label{eq: ubp}
    \frac{1}{n} \sum_{v \in \comp_{\max}}\1_{\left\{B_r^{\sss(G_n)}(v) \in \mathscr{H}_\star\right\}} \leq   Z_{\geq k, \mathscr{H}_\star} \overset{\sss\prob}{\to}\mu\left(|\mathscr{C}^{\sss-}_o|\geq k, |\mathscr{C}^{\sss+}_o|\geq k, B_r^{\sss(G)}(o) \in \mathscr{H}_\star\right).
\end{equation}
Define 
\begin{equation}
    \mu_k(\mathscr{H}_\star) = \mu\left(|\mathscr{C}^{\sss-}_o|\geq k, |\mathscr{C}^{\sss+}_o|\geq k, B_r^{\sss(G)}(o) \in \mathscr{H}_\star\right).
\end{equation}
Then, by \eqref{eq: ubp},
\begin{equation}\label{eq: upper bound on local limit of giant}
    \lim_{n\to \infty}\prob\left(\frac{1}{n} \sum_{v \in \comp_{\max}}\1_{\left\{B_r^{\sss(G_n)}(v) \in \mathscr{H}_\star\right\}} \leq  \mu_k\left(\mathscr{H}_\star\right) +\frac{\varepsilon}{2}\right) =1.
\end{equation}
Choosing $\mathscr{H}_\star = \{H_\star\}$, gives the upper bound. For, $\mathscr{H}_\star = \{H_\star\}^c$ in \eqref{eq: upper bound on local limit of giant},
\begin{equation}
    \lim_{n\to \infty} \prob\left(\frac{1}{n}\left[|\comp_{\max}| - \sum_{v\in \comp_{\max}}\1_{\{B_r^{\sss(G_n)}(v) \cong H_\star\}}\right] \leq \mu_k(\mathscr{H}_\star) +\frac{\varepsilon}{2} \right) = 1.
\end{equation}
We again argue by contradiction. For this, we assume that
\begin{equation}\label{eq: contradiction}
    \lim_{n\to \infty} \prob\left(\frac{1}{n}\sum_{v\in \comp_{\max}}\1_{\{B_r^{\sss(G_n)}(v) \cong H_\star\}}\leq \mu_k(\{H_\star\}) -\frac{\varepsilon}{2} \right) = \kappa>0.
\end{equation}
Then, with asymptotic probability $\kappa$, 
\eqan{
\frac{|\comp_{\max}|}{n} &= \frac{1}{n}\left[|\comp_{\max}| - \sum_{v\in \comp_{\max}}\1_{\{B_r^{\sss(G_n)}(v) \cong H_\star\}}\right] + \frac{1}{n}\sum_{v\in \comp_{\max}}\1_{\{B_r^{\sss(G_n)}(v) \cong H_\star\}}\nonumber\\
&\leq \zeta_{\geq k}-\vep/2 \leq \zeta-\vep/2.}
This contradicts Theorem \ref{thm: strong giant is local}. Thus, \eqref{eq: contradiction} cannot hold. This completes the proof. 
\end{proof}

\section{Number of strong components: Proof of Theorems \ref{thm: limiting number of scc} \& \ref{thm: Number of scc in locally tree-like graph}}\label{sec-number-strong-components}
In this section, we prove results concerning the number of strongly connected components of random digraph sequences having a local limit.

\subsection{Number of strong components is `almost' local: Proof of Theorem \ref{thm: limiting number of scc}}
In this section, we prove Theorem \ref{thm: limiting number of scc}. We prove a lower bound that holds more generally, and an upper bound that holds under the assumptions stated in Theorem \ref{thm: limiting number of scc}.

We first use that
    \begin{align}
    \label{number-SCCs}
    K_n = \sum_{v\in V(G_n)}\frac{1}{|\comp_v|}.
    \end{align}
Indeed, if one sums the right-hand side for the vertices as partitioned by the strongly connected components, one obtains
\begin{align}
    \sum_{v\in V(G_n)}\frac{1}{|\comp_v|} = \sum_{1 \leq i \leq K_n}\sum_{v\in \comp_{\sss(i)}}\frac{1}{|\comp_{\sss(i)}|} = \sum_{1 \leq i \leq K_n}1 = K_n. 
\end{align}

Equation \eqref{number-SCCs} is the starting point of our analysis. However, we emphasize that, unlike in the undirected case, $h(G,o)=1/|\comp_v|$ is {\em not} a bounded and continuous function in the forward-backward sense.

For the lower bound, we fix $k$ and bound
    \eqn{
    K_n \geq  \sum_{v\in V(G_n)}\frac{1}{|\comp_v|}\indicwo{\{|\mathscr{C}^{\sss-}_v|<k\}\cup \{|\mathscr{C}^{\sss+}_v|<k\}}.
    }
For each $k<\infty$, the function 
    \[h(G,o)=\frac{1}{|\comp_o|}\indicwo{\{|\mathscr{C}^{\sss-}_o|<k\}\cup \{|\mathscr{C}^{\sss+}_o|<k\}}
    \]
is a bounded continuous function. Indeed, any $u\in \comp_o$ is on a directed cycle starting and ending at $o$. When $|\mathscr{C}^{\sss-}_o|<k$ or $ |\mathscr{C}^{\sss+}_o|<k$, there cannot be any cycle longer than $k-1$, so that we can determine $h(G,o)$ on the basis of $B_k^{\sss(G)}(o)$ (recall Definition \ref{def-forward-backward-nbd}).

We conclude that
    \eqn{
    \frac{1}{n}\sum_{v\in V(G_n)}\frac{1}{|\comp_v|}\indicwo{\{|\mathscr{C}^{\sss-}_v|<k\}\cup \{|\mathscr{C}^{\sss+}_v|<k\}}
    \convp \expec_\mu\Bigg[\frac{1}{|\comp_o|}
    \indicwo{\{|\mathscr{C}^{\sss-}_o|<k\}\cup \{|\mathscr{C}^{\sss+}_o|<k\}}\Bigg].
    }
Since, for $k\rightarrow \infty$,
    \eqn{
    \expec_\mu\Bigg[\frac{1}{|\comp_o|}
    \indicwo{\{|\mathscr{C}^{\sss-}_o|<k\}\cup \{|\mathscr{C}^{\sss+}_o|<k\}}\Bigg]\rightarrow \expec_\mu\Bigg[\frac{1}{|\comp_o|}
    \indicwo{\{|\mathscr{C}^{\sss-}_o|<\infty\}\cup \{|\mathscr{C}^{\sss+}_o|<\infty\}}\Bigg],
    }
we obtain that
    \eqn{
    \frac{K_n}{n}\geq \expec_\mu\Bigg[\frac{1}{|\comp_o|}
    \indicwo{\{|\mathscr{C}^{\sss-}_o|<\infty\}\cup \{|\mathscr{C}^{\sss+}_o|<\infty\}}\Bigg]+o_{\sss\prob}(1),
    }
as required.

For the upper bound, we again fix $k$ and now write
    \eqn{
    \label{Kn-rewrite}
    K_n =  \sum_{v\in V(G_n)}\frac{1}{|\comp_v|}\indicwo{\{|\mathscr{C}^{\sss-}_v|<k\}\cup \{|\mathscr{C}^{\sss+}_v|<k\}}
    +\sum_{v\in V(G_n)}\frac{1}{|\comp_v|}\indicwo{\{|\mathscr{C}^{\sss-}_v|\geq k\}\cap \{|\mathscr{C}^{\sss+}_v|\geq k\}}.
    }
It suffices to prove that the second term in \eqref{Kn-rewrite} is $o_{k,\sss\prob}(n),$ which is what we will do now. For this, we make the further split
    \eqan{
    &\sum_{v\in V(G_n)}\frac{1}{|\comp_v|}\indicwo{\{|\mathscr{C}^{\sss-}_v|\geq k\}\cap \{|\mathscr{C}^{\sss+}_v|\geq k\}}\\
    &\qquad=\sum_{v\in V(G_n)}\frac{1}{|\comp_v|}[\indicwo{\{|\mathscr{C}^{\sss-}_v|\geq k\}\cap \{|\mathscr{C}^{\sss+}_v|\geq k\}}
    -\indic{v\in \comp_{\max}}]+\sum_{v\in V(G_n)}\frac{1}{|\comp_v|}\indic{v\in \comp_{\max}}.\nonumber
    }
We note that
    \eqn{
    \label{easy-term-SCC}
    \frac{1}{n}\sum_{v\in V(G_n)}\frac{1}{|\comp_v|}\indic{v\in \comp_{\max}}=
    \frac{1}{n}\sum_{v\in \comp_{\max}}\frac{1}{|\comp_{\max}|}
    =\frac{1}{n}\rightarrow 0.
    }
Further, since $|\comp_{\max}|\convp \infty$ by assumption, we have that $|\comp_{\max}|\geq k$ whp, so that whp also
    \[
    \indicwo{\{|\mathscr{C}^{\sss-}_v|\geq k\}\cap \{|\mathscr{C}^{\sss+}_v|\geq k\}}
    -\indic{v\in \comp_{\max}}\geq 0.
    \]
As a result, using that $|\comp_v|\geq 1,$ on the high probability event that $|\comp_{\max}|\geq k$,
    \eqan{
    \label{hard-term-SCC}
    &\frac{1}{n}\sum_{v\in V(G_n)}\frac{1}{|\comp_v|}[\indicwo{\{|\mathscr{C}^{\sss-}_v|\geq k\}\cap \{|\mathscr{C}^{\sss+}_v|\geq k\}}
    -\indic{v\in \comp_{\max}}]\\
    &\qquad\leq \frac{1}{n}\sum_{v\in V(G_n)}[\indicwo{\{|\mathscr{C}^{\sss-}_v|\geq k\}\cap \{|\mathscr{C}^{\sss+}_v|\geq k\}}
    -\indic{v\in \comp_{\max}}]\nonumber\\
    &\qquad=\frac{1}{n}\big[Z_{\geq k}-|\comp_{\max}|].\nonumber
    }
By Theorem \ref{thm: Upper bound strong giant} and forward-backward local convergence in probability of $G_n$,
    \eqn{
    \frac{1}{n}\big[Z_{\geq k}-|\comp_{\max}|]\convp \zeta_{\geq k}-\zeta,
    }
which vanishes as $k\rightarrow \infty$. We conclude that also the left-hand side of \eqref{hard-term-SCC} is $o_{k,\sss \prob}(1)$, which, together with \eqref{easy-term-SCC}, completes the proof.
\qed
\subsection{Isolated vertices in digraphs and locally tree-like digraphs}
In this section, we investigate strongly isolated vertices, which are of interest because each such vertex contributes a unique strongly connected component. In locally tree-like graphs these are the major contributors when counting the number of strongly connected components, as we will show in this section.

\begin{definition}[Strongly isolated vertices]\label{def: Strongly isolated vertices}
   \rm A vertex $v\in V(G)$ is said to be \emph{strongly isolated} if $\mathscr{C}_v = \{v\}.$ We denote the set of strongly isolated vertices of $G$ by $V_1(G)$, and denote $\alpha_1(n) = |V_1(G)|$.
\end{definition}

We next study $\alpha_1(n):$
\begin{lemma}[Upper bound for the proportion of strongly isolated vertices]\label{lem: Upper bound strongly isolated vertices}
   Under the assumptions of Theorem \ref{thm: strong giant is local}, and for all $\varepsilon >0$,
    \begin{equation}
        \lim_{n\to \infty}\prob\left(\alpha_1(n) \leq 1-\zeta + \varepsilon\right) = 1, 
    \end{equation}
    where $\zeta = \prob(|\mathscr{C}_o^{\sss-}| = |\mathscr{C}_o^{\sss+}|=\infty)$.
\end{lemma}
\begin{proof}
   The number of isolated vertices cannot exceed the number of vertices outside the giant component, since the vertices in the giant cannot be isolated. This gives $\alpha_1(n) \leq n - |\comp_{\max}|$. Dividing by $n$ on both sides gives the required result.
\end{proof}

\begin{definition}[In- and out-neighbourhood of a fixed finite radius]\label{def: In and out neighbourhoods}
\rm Let $G$ be a digraph and $v \in V(G)$, then we let $B_k^{\sss-}(v)$ and $B_k^{\sss+}(v)$ to be the subgraphs induced by the vertices in the $k$ radius in- and out-neighbourhood of $v$, respectively, defined as
\begin{align}
    B_k^{\sss-}(v) = G\left[\left\{u \in V(G) \text{ s.t. } d_G(u, v) < k\right\}\right],\\ B_k^{\sss+}(v) = G\left[\left\{u \in V(G) \text{ s.t. } d_G(v, u) < k\right\}\right].
\end{align}
Similarly, we let $\partial B_k^{\sss-}(v)$ and $\partial B_k^{\sss+}(v)$ to be the boundary of $k$ radius in- and out-neighbourhoods of $v$, respectively, defined as 
\begin{align}
    \partial B_k^{\sss-}(v) = \left\{u \in V(G) \text{ s.t. } d_G(u, v) = k\right\},\\ \partial B_k^{\sss+}(v) = \left\{u \in V(G) \text{ s.t. } d_G(v, u) = k\right\}.
\end{align}
\hfill \ensymboldefinition
\end{definition}
We next turn to locally tree-like digraphs:
\begin{lemma}[Proportion of strongly isolated vertices in locally tree-like digraphs]\label{lem: number of strongly isolated vertices in locally tree-like graphs}
    Suppose that $G_n$ is a sequence of random digraphs such that $G_n$ converges forward-backward locally in probability to $(G, o)\sim \mu$, which is almost surely a random directed tree. Then, under the assumption of Theorem \ref{thm: strong giant is local},
    \begin{equation}
        \alpha_1(n) \overset{\sss\prob}{\longrightarrow} 1-\zeta.
    \end{equation}
\end{lemma}

\begin{proof}[Proof of Lemma \ref{lem: number of strongly isolated vertices in locally tree-like graphs}]
Fix a finite $k \in \mathbb{N}$. Then
     \begin{align}
     \begin{split}
        \alpha_1(n)&= \frac{1}{n}\sum_{v \in V(G_n)}
        \mathbbm{1}_{\{|\mathscr{C}_v| = 1\}}\\
        &\geq \frac{1}{n}\sum_{v \in V(G_n)}\mathbbm{1}_{\{|B_k^{\sss-}(v)\cap B_k^{\sss+}(v)|= 1, |\partial B_k^{\sss-}(v)||\partial B_k^{\sss+}(v)| = 0\}}\\
        &\overset{\sss\prob}{\longrightarrow} \mu(|B_k^{\sss-}(o)\cap B_k^{\sss+}(o)|= 1, |\partial B_k^{\sss-}(o)||\partial B_k^{\sss+}(o)| = 0).  
     \end{split}
    \end{align}
If the local limit is a directed tree, then $\mu(|B_k^{\sss-}(o)\cap B_k^{\sss+}(o)|= 1) =1$, since otherwise there will be a directed cycle in the local limit graph. Thus, we can write 
    \eqan{
   &\mu(|B_k^{\sss-}(o)\cap B_k^{\sss+}(o)|= 1, |\partial B_k^{\sss-}(o)||\partial B_k^{\sss+}(o)| = 0)\\
   &\qquad= \mu(|\partial B_k^{\sss-}(o)||\partial B_k^{\sss+}(o)| = 0).\nonumber
    }
Also, we know that, as $k\to \infty$,
\begin{equation}
    \mu( |\partial B_k^{\sss-}(o)||\partial B_k^{\sss+}(o)| = 0) \to \mu\left(\left\{|\mathscr{C}_o^{\sss-}|<\infty\right\} \cup \left\{|\mathscr{C}_o^{\sss+}|<\infty\right\}\right) = 1- \zeta.
\end{equation}
This completes the proof.
\end{proof}

\begin{proof}[Proof of Theorem \ref{thm: Number of scc in locally tree-like graph}] By the definition of $\alpha_1(n)$ in Definition \ref{def: Strongly isolated vertices},
\begin{equation}\label{eq: lower bound no. scc}
        \frac{K_n}{n}\geq \alpha_1(n), 
\end{equation}
 Also,    
\begin{equation}\label{eq: upper bound no. scc}
        K_n \leq 1 + n - |\mathscr{C}_{\max}|.
\end{equation}
From \eqref{eq: lower bound no. scc} and \eqref{eq: upper bound no. scc}, 
\begin{equation}\label{eq: bounds no. scc}
    \alpha_1(n) \leq \frac{K_n}{n} \leq 1 - \frac{|\mathscr{C}_{\max}|}{n}+\frac{1}{n}. 
\end{equation}
Lemma \ref{lem: number of strongly isolated vertices in locally tree-like graphs} and Theorem \ref{thm: strong giant is local} then complete the proof.
\end{proof}

\section{Discussion}\label{sec-discussion}

\subsection{Discussion results and literature} The main results in this paper concern the proportion of vertices in the LSCC and the number of connected components. We prove that these properties are `almost local'. Why are these connectivity properties important? Apart from indicating the connectivity of the graph, they are also used in the comparisons of centrality measures (see, e.g., \cite{mocanu2018decentralized}). This investigates how much a graph breaks down in terms of its connectivity, when we remove the most central nodes. We next list some open problems of our work.

In Theorem \ref{thm: limiting number of scc}, we prove that \eqref{eq: condition for giant} is a necessary and sufficient condition  for $|\mathscr{C}_{\max}|/n\convp \zeta=\mu(|\mathscr{C}^{\sss+}_o|=|\mathscr{C}^{\sss-}_o|=\infty)$ to hold. Of course, $\zeta\geq \zeta'\equiv \mu(|\mathscr{C}_o|=\infty)$. What is a necessary and sufficient condition for $|\mathscr{C}_{\max}|/n\convp \zeta'$ to hold? These conditions should really be different, since, for example, for locally tree-like digraphs, $\zeta'=0,$ while $|\mathscr{C}_{\max}|/n\convp \zeta=\mu(|\mathscr{C}^{\sss+}_o|=|\mathscr{C}^{\sss-}_o|=\infty)>0$ {\em can} hold.

Currently, to describe the number of strongly connected components in Theorem \ref{thm: limiting number of scc}, we rely on the necessary and sufficient condition \eqref{eq: condition for giant} for the proportion of vertices in the giant being equal to $\zeta=\mu(|\mathscr{C}^{\sss+}_o|=|\mathscr{C}^{\sss-}_o|=\infty)$. What is the necessary and sufficient condition for the convergence $K_n/n \overset{\sss\prob}{\to}\alpha$ to hold, and is it possible that $\alpha$ takes a different form than that on the right-hand side of \eqref{eq:number sccs almost local}?

\subsection{Large weak components}\label{sec-large-weak-components}
In this section, we discuss weak giants. The definition of a strong component in a digraph is well accepted. In contrast, the notion of a weak component has been used differently in the literature, as we discuss next.

\paragraph{\bf In- and out-components.} One could consider weak components to be in-and the out-components. For example, in \cite{MR2056402}, the authors study the largest in- and out-component sizes for the directed configuration model. We define the sizes of the largest in- and out-components $\mathscr{I}_{\max}$ and $\mathscr{O}_{\max}$ to be 
\begin{align}
    \mathscr{I}_{\max} = \max_{v\in V(G_n)}|\mathscr{C}^{\sss-}_v| \qquad \text{and} \qquad \mathscr{O}_{\max} = \max_{v\in V(G_n)}|\mathscr{C}^{\sss+}_v|.
\end{align}
The strong giant having linear size implies that its in- and out-components also have linear size, but the converse is not necessarily true. We visualize this in Figure \ref{Figure: representation}. 
\begin{figure}[ht!]
\centering
\begin{tikzpicture}[-latex ,auto ,node distance =1.5 cm and 2cm ,on grid ,
semithick ,
state/.style ={ circle ,top color =white , bottom color = processblue!20 ,
draw,processblue , text=blue , minimum width =0.2 cm}]
     \draw  \firstcircle node[above] {$\mathscr{C}_{\sss(1)}$};
     \draw [-stealth] (1,0) -- (2,0);
     \draw  \secondcircle node [above] {$\mathscr{C}_{\sss(2)}$}; 
   \node at (4.5, 0) {\ldots};
   \draw  \thirdcircle node [above] {$\mathscr{C}_{\sss(k-1)}$};
   \draw [-stealth] (7,0) -- (8,0);
     \draw  \fourthcircle node [above] {$\mathscr{C}_{\sss(k)}$};
\end{tikzpicture}
\caption{\small A linear-sized giant in- or out-component can exist despite the fact that there is no linear-sized strong giant.}\label{Figure: representation}
\end{figure}
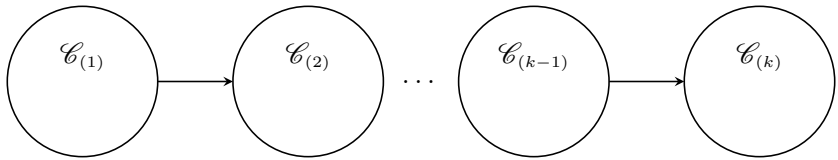
The big question is whether the maximal in- and out-components are `almost' local or not, and if so, under what condition. In particular, a natural question is to find a condition that guarantees that the proportion of vertices in $\mathscr{I}_{\max}$ is asymptotically equal to the proportion of vertices in the out-component of $\mathscr{C}_{\max}$. See Figure \ref{Fig: bowtie}, where LSCC stands for the largest strongly connected component. This means that the condition for the existence of weak giants should be weaker than the conditions given in Theorem \ref{thm: strong giant is local}.

\begin{figure}
\centering
\begin {tikzpicture}[-latex ,auto ,node distance =1.5 cm and 2cm ,on grid ,
semithick ,
state/.style ={ circle,top color =white , bottom color = black!20 ,
draw, black , text=blue , minimum width =0.2 cm}]
\draw \firstcircle node[above] {LSCC};
\draw (-0.866, -0.5) -- (-3, -1.5) -- (-3,1.5) -- (-0.866, 0.5);
\draw (-3, -1.5) -- (-0.866, -0.5);
\draw (-3, 0) -- (-1, 0) node[pos = 0.5]{\textit{Out}};
\draw  (0.866, -0.5) -- (3, -1.5);
\draw (3, 1.5) -- (3,-1.5);
\draw (0.866, 0.5)--(3, 1.5);
\draw (1, 0) -- (3, 0) node [pos = 0.5] {\textit{In}};
\end{tikzpicture}
\caption{\small The `bow-tie' figure used to describe the connectivity structure of large directed networks.}
\label{Fig: bowtie} 
\end{figure}
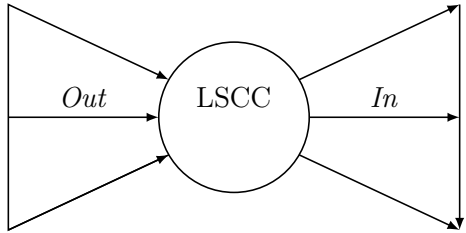
 For the in- and out-components, we believe it to be possible that $|\comp_{\max}^+|/n \overset{\sss\prob}{\to} \alpha \leq \mu(|\comp_o^-| = \infty)$ and vice-versa for the in-component. We state the following open problem:
\begin{openprob}
Find the necessary and sufficient conditions  for 
     \begin{align}
     \frac{|\comp_{\max}^+|}{n} \overset{\sss\prob}{\to} \alpha^+ \qquad \text{and} \qquad \frac{|\comp_{\max}^-|}{n} \overset{\sss\prob}{\to} \alpha^- 
     \end{align}
to hold.
\end{openprob}
When $|\comp_{\max}^+|/n \overset{\sss\prob}{\to} \alpha^+$, then one can easily deduce that $\alpha^+ \leq \mu(|\comp_o^-| = \infty)$. It is not clear when $\alpha^+ =\mu(|\comp_o^-| = \infty)$ holds.
One difficultly to prove such a result is that, unlike the strongly connected components, the in- and out-components do not \emph{partition} the vertex set.

\paragraph{\bf Ignoring the directions of edges.} 
One can also consider weak components to consist of the set of vertices that can be reached by recursively following all edges regardless of their orientations \cite{coulson2021weak, kryven2016emergence}. This corresponds to the connected components in the underlying undirected graph, i.e., the graph obtained after removing the edge directions. These types of components find applications in real-world problems, including epidemiology, data mining, communication networks, world wide web structure, etc. In this case, the results in \cite{Hofs21b} can be used to give a necessary and sufficient condition for the proportion of vertices in the giant to be the survival probability of the local limit.

\paragraph{\bf Weak components as a relation.} The 1972 paper by Graham, Knuth, and Motzkin \cite{donald1999art, graham1972complements} introduced a rather different notion of weak components. Formally, it can be defined by the following $4$ symmetric relations on the vertex set of a digraph: 
\begin{enumerate}
    \item For any two vertices $u$ and $v$ in $V(G)$, $u \iff v$ if both vertices are in the same strongly connected component, i.e., $\max\{d_G(u, v), d_G(v, u)\} < \infty$. In other words, there exists a path from $v$ to $u$ and vice versa.
    \item For any two vertices $u$ and $v$ in $V(G)$, $u\parallel v$ if $\min\{d_G(u, v), d_G(v, u)\} =\infty$, i.e., there is no path from $v$ to $u$ and from $u$ to $v$.
    \item For any two vertices $u$ and $v$ in $V(G)$, $u \approx v$ if either $u\iff v$ or $u \parallel v$.
    \item We define $R$ to be the transitive closure of $\approx$. This means that if there is a chain $(u=v_0, v_1, \ldots, v_k = v)$ such that $u\approx v_1 \approx  \cdots v_{k-1} \approx v$, then $uRv$.
\end{enumerate}
The relation $R$ is an equivalence relation and partitions the vertex set into the weak components. This concept of weak component is not a local quantity as a small perturbation in the graph can easily dramatically change the components in the graph. In particular, the addition of \emph{one} isolated vertex can change the definition of all these weak components. This suggests that, in terms of local convergence, this notion is too sensitive to local changes. Yet,
Pacault \cite{pacault1974computing} gives an algorithm for computing weak components. 

\bibliographystyle{amsplain}
\bibliography{bib}




\end{document}